\documentclass[a4paper]{amsart}

\usepackage{a4wide}
\usepackage{amssymb}
\usepackage{amsthm}

\title{The symmetries of the ${}_2\phi_1$}

\author{Fokko van de Bult}

\date{\today}

\newcommand{\mfi}[4]{\;{}_2\phi_1(#1,#2;#3;q,#4)}

\newtheorem{theorem}{Theorem}[section]
\newtheorem{defi}[theorem]{Definition}
\newtheorem{prop}[theorem]{Proposition}
\newtheorem{lemma}[theorem]{Lemma}
\newtheorem{cor}[theorem]{Corollary}

\newenvironment{acknowledgements}{\textbf{Acknowledgements:}}{}

\begin{document}

\begin{abstract}
We show that the only symmetries of the ${}_2\phi_1$ within a large class of
possible transformations are Heine's transformations. 
The class of transformations considered consists of equation of the form
${}_2\phi_1(a,b;c;q,z)= f(a,b,c,z)\; {}_2\phi_1(L(a,b,c,q,z))$, where $f$ is a 
$q$-hypergeometric term and $L$ a linear operator on the logarithms of the 
parameters. We moreover prove some results on $q$-difference equations satisfied
by ${}_2\phi_1$, which are used to prove the main result.
\end{abstract}

\maketitle

Basic hypergeometric series and their properties have been studied a long time, as they have many applications.
One of the interesting properties of such series is the fact that they satisfy certain transformation formulas.
We are interested in transformation formulas of one series to a similar series with different parameters.
These symmetries of hypergeometric series often arise from symmetries of underlying algebraic structures.

In many important cases transformation formulas are well-known, one of the simplest of which is Heine's transformation
of the ${}_2\phi_1$. The ${}_2\phi_1$ is a basic hypergeometric analogue of the Gau\ss\ hypergeometric function
${}_2F_1$. In this article we use the notation from Gasper and Rahman \cite{GenR}.
We assume $q\in \mathbb{C}^*$ is some complex number satisfying $|q|<1$, which ensures all infinite products converge.
We define $q$-shifted factorials by
\[
(x;q)_k= \frac{(x;q)_\infty}{(xq^k;q)_\infty}, \qquad (x;q)_\infty = \prod_{j=0}^\infty (1-xq^j),
\]
and use the abbreviated notation $(a_1,a_2,\ldots,a_s;q)_k = \prod_{j=1}^s (a_j;q)_k$.
The series ${}_2\phi_1$ is now given by the series
\[
{}_2\phi_1(a,b;c;q,z) = \sum_{k=0}^\infty \frac{(a,b;q)_k}{(q,c;q)_k} z^k.
\]
The ${}_2\phi_1$ converges if and only if $|z|<1$. For arbitrary $|q|<1$, the ${}_2\phi_1$ can be analytically extended to a
meromorphic function for $(a,b,c,z) \in \mathbb{C}^4$, and we will denote this extention simply by ${}_2\phi_1$.

Heine's transformation \cite[(III.1)]{GenR}  gives the equality
\begin{equation}\label{eqHeine}
{}_2\phi_1(a,b;c;q,z) = \frac{(b,az;q)_\infty}{(c,z;q)_\infty} {}_2\phi_1(c/b,z;az;q,b).
\end{equation}
This symmetry together with the permutation symmetry of $a$ and $b$ generates a symmetrygroup 
isomorphic to $S_3\times S_2$ as already noted by Rogers \cite{Rogers}. 
It is commonly assumed that Heine's transformations and its iterates exhaust all the symmetries of the 
${}_2\phi_1$, however this has not yet been proven rigorously.
In this article we prove that there exist no other such symmetries, at least not
within a certain large class of possible transformations (defined below). 
As of now we will use the word transformation for an element of the class of allowed transformation, and 
the word symmetry for an identity relating a ${}_2\phi_1$ with a transformed ${}_2\phi_1$.

The proof of the main theorem was inspired by the ideas of Petkov\v sek, Wilf and Zeilberger \cite{aisb} on evaluating hypergeometric sums.
Essentially they solve the problem of evaluating a hypergeometric sum in two steps, first they find a recurrence relation satisfied by 
this sum, and subsequently they find the solutions to this recurrence relation. Within these solutions it is
then easy to find the right evaluation of the original sum (if it exists). 

The main idea of the proof in this article is similar. We first consider the contiguous relations satisfied by 
the ${}_2\phi_1$. Subsequently we find all transformed ${}_2\phi_1$'s which satisfy these contiguous relations.
From these the symmetries of the ${}_2\phi_1$ are easily determined.


The method used in this article to prove that the known symmetries of the ${}_2\phi_1$ are indeed all the 
symmetries it satisfies can probably also be used to give a complete classification of the symmetries of other 
basic hypergeometric series, such as balanced ${}_3\phi_2$'s and very-well poised ${}_8\phi_7$'s.
An essential part of the proof uses that we can relate different parameters by the known symmetries
(for example Heine's symmetry \eqref{eqHeine} relates the parameter $b$ and $z$), to obtain a full set
of contiguous relations. If there do not exist such symmetries a balancing condition might be sufficient as well. 

The setup of this article is as follows. In the next section we introduce some algebraic constructs which 
allow us to give precise definitions of transformations. In particular it allows us to state the main theorem.
In Section \ref{secdiff} we consider the theory of $q$-difference equations (or contiguous relations) satisfied by the ${}_2\phi_1$.
Finally in Section \ref{secproofmain} we prove the main theorem.

\begin{acknowledgements} 
I would like to thank Jasper Stokman for his helpful comments and proofreading the article.
\end{acknowledgements}

\section{Some algebraic constructs}\label{sec1}
In order to be able to give a precise and concise statement of the main theorem and the results leading 
up to the main theorem it is convenient to introduce some notation.
In this section we define the group of transformations of a ${}_2\phi_1$ and
give an algebraic description of $q$-difference equations.

The relations between the different objects can get somewhat confusing, hence I included Figure \ref{figconst}, which
contains all objects in a single table.

The general form of Heine's symmetry is of multiplication by a $q$-hypergeometric term (see
Definition \ref{defqhyp} below) and applying a linear transformation on the logarithms of the parameters.
Let us first define this group of linear transformations.
\begin{defi}
The group $G$ is given by
\[
G=\{M\in GL_5(\mathbb{Z}) ~|~  M^T e_5=e_5\},
\]
where $e_5$ denotes the fifth unit vector and $M^T$ denotes the 
transpose of $M$.
\end{defi}
Note that the condition $M^T e_5=e_5$ signifies that the bottom row of the matrices $M\in G$ is equal 
to $(0,0,0,0,1)$. Now we give the action of $G$ on functions of 5 variables.
Let $\mathcal{M}$ be the field of meromorphic functions on the variables $(a,b,c,q,z)$.
\begin{lemma}
There exists an action of $G$ on $(\mathbb{C}^*)^5$ (elements of which are denoted by $(a,b,c,z,q)$) given by
\[
\left( 
\begin{tabular}{ccccc}
$l_{11}$ & $l_{12}$ & $l_{13}$ & $l_{14}$ & $l_{15}$ \\
$l_{21}$ & $l_{22}$ & $l_{23}$ & $l_{24}$ & $l_{25}$ \\
$l_{31}$ & $l_{32}$ & $l_{33}$ & $l_{34}$ & $l_{35}$ \\
$l_{41}$ & $l_{42}$ & $l_{43}$ & $l_{44}$ & $l_{45}$ \\
0 & 0 & 0 & 0 & 1 
\end{tabular}
\right)
\left( 
\begin{tabular}{c}
$a$ \\ $b$\\$c$\\$z$ \\ $q$
\end{tabular}
\right)
=
\left( 
\begin{tabular}{c}
$a^{l_{11}}b^{l_{12}}c^{l_{13}}z^{l_{14}}q^{l_{15}}$ \\ 
$a^{l_{21}}b^{l_{22}}c^{l_{23}}z^{l_{24}}q^{l_{25}}$ \\
$a^{l_{31}}b^{l_{32}}c^{l_{33}}z^{l_{34}}q^{l_{35}}$ \\
$a^{l_{41}}b^{l_{42}}c^{l_{43}}z^{l_{44}}q^{l_{45}}$ \\
$q$
\end{tabular}
\right).
\]
This action extends to an action of $G$ on $\mathcal{M}$ by field automorphisms given by 
$L(f)(a,b,c,z,q) = f(L^{-1}(a,b,c,z,q))$.
\end{lemma}
\begin{proof}
The proposed action of $G$ on sets of variables $(a,b,c,z,q)$ is just the normal action of 
$GL_5(\mathbb{Z})$ on 5-dimensional vectors applied to the logarithms of 
the variables, thus it is clearly a well-defined action.
\end{proof}
Note that this describtion is very similar to the use of homogeneous coordinates to express affine transformations.
The affine part of the transformations allows multiplication of some variables by powers of $q$. Such a multiplication 
does not occur in Heine's symmetry \eqref{eqHeine}, but it does happen in symmetries of other 
basic hypergeometric series, as for example the symmetries of the very well poised ${}_8\phi_7$.

As an example we see that
\[
\left( 
\begin{tabular}{ccccc}
0 & -1 & 1 & 0 & 0 \\
0 & 0 & 0 & 1 & 0 \\
1 & 0 & 0 & 1 & 0 \\
0 & 1 & 0 & 0 & 0 \\
0 & 0 & 0 & 0 & 1 
\end{tabular}
\right)
{}_2\phi_1 \left( \begin{array}{c} a,b \\ c \end{array} ;q,z \right) = 
{}_2\phi_1 \left( \begin{array}{c} c/b,z \\ az \end{array};q,b \right)
\]
provides exactly the change in arguments of the ${}_2\phi_1$ in Heine's symmetry.

Another important example shows pure $q$-dilations are contained in $G$, for example
\[
 \left( 
\begin{tabular}{ccccc}
1 & 0 & 0 & 0 & -1 \\
0 & 1 & 0 & 0 & 0 \\
0 & 0 & 1 & 0 & 1 \\
0 & 0 & 0 & 1 & -2 \\
0 & 0 & 0 & 0 & 1 
\end{tabular}
\right).
f(a,b,c,q,z) = f(q a,b,q^{-1}c,q,q^{2}z).
\]
This leads us to define the subgroup of $q$-dilations in $G$.
\begin{defi}\label{defABCZ}
We define the group $N= \{L\in G ~|~ L e_j = e_j (j=1,2,3,4)\}$. 
More explicitely $N$ consists of the matrices in $GL_5(\mathbb{Z})$ of the form
\[
M_{k_a,k_b,k_c,k_z} :=
\left( 
\begin{tabular}{ccccc}
1 & 0 & 0 & 0 & $-k_a$ \\
0 & 1 & 0 & 0 & $-k_b$ \\
0 & 0 & 1 & 0 & $-k_c$ \\
0 & 0 & 0 & 1 & $-k_z$ \\
0 & 0 & 0 & 0 & 1 
\end{tabular}
\right).
\]
The group $N$ is generated by the elements $A=M_{1,0,0,0}$, $B=M_{0,1,0,0}$, $C=M_{0,0,1,0}$ and $Z=M_{0,0,0,1}$. 
\end{defi}
Note that $A$ acts as a $q$-dilation in $a$, i.e.\ $Af(a,b,c,q,z)=f(aq,b,c,q,z)$. Similarly
$B$, $C$, respectively $Z$ act as elementary $q$-dilations in the variables $b$, $c$, respectively $z$.
\begin{lemma}
The subgroup $N$ is normal in $G$. Moreover $N$ is isomorphic to $\mathbb{Z}^4$ (and thus abelian).
\end{lemma}
\begin{proof}
Note that for $L\in G$ we have \[\langle L e_j,e_5 \rangle = \langle e_j, L^T e_5 \rangle = \langle e_j,e_5 \rangle=0\] 
for $j=1,2,3,4$, i.e. $L$ maps the space $V$ spanned by $e_1$, $e_2$, $e_3$ and $e_4$ to itself. 
Thus for any $P\in N$ and $j=1,2,3,4$ we find that $LPL^{-1} e_j = L L^{-1} e_j = e_j$, as $P$ preserves $V$ and
$L^{-1} e_j \in V$. Thus for any element $P\in N$, the conjugate $LPL^{-1}$ of $P$ with an element $L\in G$
is also contained in $N$. Therefore $N$ is a normal subgroup of $G$.

An isomorphism of $N$ to $\mathbb{Z}^4$ is given by $M_{k_a,k_b,k_c,k_z} \mapsto (k_a,k_b,k_c,k_z)$.
\end{proof}

We can now give a concise definition of $q$-hypergeometric terms.
\begin{defi}\label{defqhyp}
Define $\mathcal{R}=\mathbb{C}(a,b,c,q,z) \subset \mathcal{M}$ to be the field of rational function in $a$, $b$, $c$, $q$ and $z$.
The group $H$ of $q$-hypergeometric terms is given by
\[
 H = \{ f\in \mathcal{M}^* ~|~ P(f)/f \in \mathcal{R} \text{ for all } P \in N\}.
\]
Here a star denotes the group of units of a ring, thus $\mathcal{M}^*$ is the group of meromorphic functions, which are not
identically zero. The group action on $H$ is given by multiplication. 
\end{defi}
Typical examples of $q$-hypergeometric terms are products of terms of the form $(x;q)_\infty$, 
where $x\in \mathbb{C}[a^{\pm 1},b^{\pm 1},c^{\pm 1},q^{\pm 1},z^{\pm 1}]^*$, is some monomial in $a$, $b$, $c$, $q$ and $z$.
In particular the pre-factor $(b,az;q)_\infty/(c,z;q)_\infty$ 
in Heine's symmetry \eqref{eqHeine} is a $q$-hypergeometric term. Moreover it is clear that all non-zero rational functions
are $q$-hypergeometric terms, i.e.\ $\mathcal{R}^* \subset H$.

\begin{lemma}
The action of $G$ on $\mathcal{M}$ preserves $H$ and $\mathcal{R}$, 
i.e.\ if $L\in G$ and $h\in H$, then
also $L(h) \in H$, and similarly for $\mathcal{R}$. 
Moreover the group $G$ acts by group automorphisms on $H$ and by field automorphisms on 
$\mathcal{M}$ and $\mathcal{R}$.
\end{lemma}
\begin{proof}
As $G$ only changes the variables in which a function is evaluated, 
it is clear that it commutes with adding or multiplying functions, whether
they are elements of $\mathcal{M}$, $H$ or $\mathcal{R}$. Thus $G$ acts on 
$\mathcal{M}$ by field automorphisms 

The fact that $G$ preserves rational functions follows from the fact that $G$ sends monomials to monomials. As for $\mathcal{M}$
we find that $G$ acts by field automorphisms. 

Since $N$ is a normal subgroup of $G$ we find that 
$P(L(f))/L(f) = L(P'(f))/L(f) = L( P'(f)/f)$, for $P'=L^{-1}PL \in N$. As $G$ preserves $\mathcal{R}$ this 
shows that if $f\in H$, so is $L(f)$. As before we find that $G$ acts by group automorphisms.
\end{proof}

Now we can define the group of transformations.
\begin{defi}
With the action of $G$ on $H$ as above we define the 
group $T$ of transformations as the semi-direct product $T:=H \rtimes G$.
It thus consists of elements $fL$, $f\in H$, $L\in G$, with multiplication given by
$f_1L_1 \cdot f_2L_2 = (f_1L_1(f_2)) (L_1L_2)$.

As a notational convention, when we write something like $fL\in T$ we imply that 
$f \in H$ and $L\in G$.

The subgroup $W \subset T$ is given by those transformations which leave
${}_2\phi_1$ invariant, or the symmetries of ${}_2\phi_1$, i.e. 
\[
W=\{t\in T ~|~ t {}_2\phi_1(a,b,c,q,z)={}_2\phi_1(a,b,c,q,z)\}.
\]
\end{defi}
Our goal is to determine $W$, the elements of which correspond to transformation
formulas of ${}_2\phi_1$. 
Two special transformations are given by
\begin{equation}\label{eqelw}
 t_h:=\frac{(b,az;q)_\infty}{(c,z;q)_\infty} 
\left( \begin{array}{ccccc}
  0 & -1 & 1 & 0 & 0 \\ 
  0 & 0 & 0 & 1 & 0 \\
  1 & 0 & 0 & 1 & 0 \\
  0 & 1 & 0 & 0 & 0 \\
  0 & 0 & 0 & 0 & 1 
\end{array} \right), \qquad 
\text{and} \qquad
t_{ab}:=\left( \begin{array}{ccccc}
  0 & 1 & 0 & 0 & 0 \\ 
  1 & 0 & 0 & 0 & 0 \\
  0 & 0 & 1 & 0 & 0 \\
  0 & 0 & 0 & 1 & 0 \\
  0 & 0 & 0 & 0 & 1 
\end{array} \right).
\end{equation}
Indeed Heine's transformation \eqref{eqHeine} exactly says $t_h {}_2\phi_1 ={}_2\phi_1$, so $t_h \in W$.
Moreover $t_{ab}$ just interchanges $a$ and $b$, thus 
$t_{ab} {}_2\phi_1(a,b;c;q,z) = {}_2\phi_1(b,a;c;q,z)={}_2\phi_1(a,b;c;q,z)$ implies $t_{ab} \in W$.

Let us now state the main theorem 
\begin{theorem}\label{thmain}
The symmetry group $W$ of the ${}_2\phi_1$ is generated by $t_h$ and $t_{ab}$ from \eqref{eqelw}, i.e.\ Heine's 
symmetry and the symmetry which interchanges $a$ and $b$ in the ${}_2\phi_1$.
\end{theorem}
The proof of this theorem is deferred to the end of this article (Section \ref{secproofmain}).
As it is already well-known that $t_h,t_{ab}\in W$, we only have to show that they generate $W$.
Note that the theorem implies that $W$ is isomorphic to $S_3 \times S_2$ and observe that both $t_h$ and $t_{ab}$ have order 2.

Let us end this section by introducing the ring of $q$-difference operators.
\begin{defi}
The ring $\mathcal{D}$ of $q$-difference operators is defined as the smash product $\mathcal{R} \# N$, using the 
action of $N$ on $\mathcal{R} \subset \mathcal{M}$ (which is the restriction of the action of $G$).

Thus as a vectorspace it equals $\mathcal{R} \otimes_{\mathbb{C}} \mathbb{C}[N]$.
Multiplication is defined via
\begin{equation}\label{eqprod2}
 (\sum_k r_k P_k) ( \sum_j s_j Q_j) = \sum_{k,j} r_k P_k(s_j) P_kS_j,
\end{equation}
(where the term $P_k(s_j)$ uses the action of $N$ on $\mathcal{R}$).

Finally we define the ideal $\mathcal{I} = \{ D\in \mathcal{D} ~|~ D{}_2\phi_1 =0\}$ of elements which 
contain ${}_2\phi_1$ in their kernel.
\end{defi}
Note that elements of $\mathcal{I}$ correspond to $q$-difference equations satisfied by ${}_2\phi_1$.
We will give a precise description of all the elements in $\mathcal{I}$ at a later stage. 
An example of the action of a $q$-difference operator is 
\begin{multline*}
 ((1-a) A - 1 + aZ) {}_2\phi_1\left( \begin{array}{c} a,b\\c\end{array};q,z\right) \\ = 
(1-a) {}_2\phi_1\left( \begin{array}{c} aq,b\\c\end{array};q,z\right)
-{}_2\phi_1\left( \begin{array}{c} a,b\\c\end{array};q,z\right)
+a{}_2\phi_1\left( \begin{array}{c} a,b\\c\end{array};q,qz\right).
\end{multline*}
In Proposition \ref{prop1} we show that this expression equals zero, thus $(1-a)A-1+aZ \in \mathcal{I}$.

Just as $N$ is a normal subgroup of $G$ we find 
\begin{lemma}
The group $\mathcal{D}^*$ of units in $\mathcal{D}$ equals the semi-direct product $\mathcal{R}^* \ltimes N$.
$\mathcal{D}^*$ is a normal subgroup of $T$. The induced conjugation action is denoted by
$t(D)= t D t^{-1}$ for $D\in \mathcal{D}^*$, $t\in T$ and can be uniquely linearly extended to an action of $t$ on $\mathcal{D}$.
\end{lemma}
\begin{proof}
To see that $\mathcal{D}^*= \mathcal{R}^* \ltimes N$ we define the degree of 
some element of $\mathcal{R}^* \ltimes N$ by using the isomorphism between $N$ and $\mathbb{Z}^4$ and using
the lexicographical ordering on $\mathbb{Z}^4$. Thus
$deg(M_{k_a,k_b,k_c,k_z})=(k_a,k_b,k_c,k_z)$ and $deg(r)=(0,0,0,0)$ for some rational function $r$.
Moreover we set $(k_a,k_b,k_c,k_z)>(l_a,l_b,l_c,l_z)$ if $k_a>l_a$, or $k_a=l_a$ and $k_b>l_b$, etc.. 

We want to show that the product of two elements of $\mathcal{D}$ can only equal 1 if they are contained
in $\mathcal{R}^* \rtimes N$. 
Note that the right hand side of \eqref{eqprod2} has a unique term of highest degree (the product of the two respective terms
of highest degree on the left hand side) and a unique term of lowest degree. The product can only be 1 if there is just
one term on the right hand side, so then the terms of highest and lowest degree are identical, 
which can only happen if the sums on the left hand side are both single terms. Thus we find $\mathcal{D}^*=\mathcal{R}^* \ltimes N$.

We calculate for $hL\in T$ and $D=rP\in \mathcal{D}^*$ that
\[hL \cdot rP \cdot (hL)^{-1} = h  L(r) L P L^{-1}( h^{-1}) LPL^{-1} \]
and use that $N$ is a normal subgroup of $G$ to see that $\mathcal{D}^*$ is a normal subgroup of 
$T$ (where we observe that $h  LPL^{-1}(h^{-1}) = h/LPL^{-1}(h)$ is a rational function as 
$h$ is a $q$-hypergeometric term).
\end{proof}
This provides us with the following important corollary, which allows us to easily obtain new $q$-difference equations 
satisfied by the ${}_2\phi_1$ from old ones by using symmetries.
\begin{cor}\label{cormain}
The restriction of the conjugation action of $T$ on $\mathcal{D}$ to $W$ 
leaves the ideal $\mathcal{I}$ invariant, i.e. for $w\in W$ and $I \in \mathcal{I}$ we have 
$w(I) \in \mathcal{I}$.
\end{cor}
\begin{proof}
Just observe that for $w\in W$ and $I\in \mathcal{I}$ we find
\[ w(I) {}_2\phi_1 = w(I(w^{-1}({}_2\phi_1)))= w(I({}_2\phi_1))=w(0)=0. \qedhere\]
\end{proof}

\begin{figure}
\begin{tabular}{ccccc}
 & Difference &  & Transformations  & \\
  & operators &  &                    & \\
Operations on coefficients & $N$ & $\lhd$ & G & \\
&  $\qquad \;\;\;$    \# $\qquad \ltimes$ &   & $\ltimes$ &   \\
Functions as prefactor & $\mathcal{R}$ & $\mathcal{R}^* \subset$ & $H$ & $\subset \mathcal{M}^*$ \\
 & $=$ & $= \quad $  & $=$ &  \\
Full algebras & $\mathcal{D}$ & $ \mathcal{D}^* \lhd$ & $T$ & \\
& $\cup$ &  & $\cup$ & \\
Symmetries of ${}_2\phi_1$ & $\mathcal{I}$ & & $W$ & 
\end{tabular}
\label{figconst}
\caption{The algebraic constructs of Section \ref{sec1} and their relations.}
\end{figure}

\section{Difference operators}\label{secdiff}
In this section we consider $q$-difference equations satisfied by the ${}_2\phi_1$. This means 
we will study the structure of the ideal $\mathcal{I} \subset \mathcal{D}$.
The main result (Proposition \ref{prop3} together with Corollary \ref{cor3}) 
is that for any three distinct elements $X_1,X_2,X_3 \in N$ there exists a unique (up to multiplication by rational functions) 
element in $\mathcal{I}$ which is a linear combination of these $X_j$.
This corresponds to a unique linear relation between three ${}_2\phi_1$'s with $q$-shifted 
arguments. In particular it implies that ${}_2\phi_1$ is a solution to second order $q$-difference equations
in all its parameters. This fact is already known (see for example \cite{Vidunas}), however we have been unable to 
find proofs of this result in the literature, which are strong enough for our purposes, hence we give full proofs here.

Let us now define some elements in $\mathcal{I}$ which turn out to be a generating set of $\mathcal{I}$
in Corollary \ref{cor3}. Essentially the elements mentioned here correspond to $q$-contiguous relations
satisfied by the ${}_2\phi_1$. 
Recall Definition \ref{defABCZ}.  
\begin{prop}\label{prop1}
Define
\begin{align*}
P_a&:= (1-a)A - 1 + a Z, \\
P_b&:= (1-b)B - 1 + b Z, \\
P_c&:= z(c-b)(c-a)C +(c-1)(c^2+abz-(a+b) cz)  - (c-1)c (c-abz) Z,\\
Q_a&:= (-cq+ac+aq-a^2z) + a(abz-c)Z  + q (c-a) A^{-1},\\
Q_b&:= (-cq+bc+bq-b^2z) + b(abz-c)Z  + q (c-b) B^{-1},\\
Q_c&:= -q  + c Z + (q-c)C^{-1},\\
R_z&:= -(c+q-az-bz)  + (q-z) Z^{-1} + (c-abz)Z.
\end{align*}
Then $P_a,P_b,P_c,Q_a,Q_b,Q_c,R_z\in \mathcal{I}$.
\end{prop}
\begin{proof}
We first explicitly prove two simple $q$-difference equations satisfied by ${}_2\phi_1$.
Subsequently we will use Corollary \ref{cormain}, 
together with the known fact that $t_h$ and $t_{ab}$ are elements of $W$
to find the other relations. 

We will first show that $P_a\in \mathcal{I}$. Observe that 
\[
(1-a) \frac{(qa,b;q)_k}{(q,c;q)_k}z^k - \frac{(a,b;q)_k}{(q,c;q)_k}z^k + 
a \frac{(a,b;q)_k}{(q,c;q)_k} (qz)^k =0,
\]
which can be proven by simply dividing everything by $\frac{(a,b;q)_k}{(q,c;q)_k}z^k$ and
checking that the resulting polynomial equation holds. Summing this 
equation over all $k\geq 0$ we obtain the relation
\begin{equation*}
0= (1-a)\mfi{aq}{b}{c}{z} - \mfi{a}{b}{c}{z} + a \mfi{a}{b}{c}{qz}.
\end{equation*}
This equation is equivalent to the statement $P_a =  (1-a) A - 1 + a Z \in \mathcal{I}$.
Using the $a\leftrightarrow b$ symmetry $t_{ab}$ of $\mfi{a}{b}{c}{z}$, we find that $P_b=t_{ab}(P_a) \in \mathcal{I}$ as well.
To show that $Q_c\in \mathcal{I}$ we note that
\[
-q \frac{(a,b;q)_k}{(q,c;q)_k} z^k + c \frac{(a,b;q)_k}{(q,c;q)_k}q^kz^k + (q-c) \frac{(a,b;q)_k}{(q,c/q;q)_k}z^k=0 
\]
for all $k$, and sum over all $k\geq 0$.

Recall that $t_h$ from \eqref{eqelw} is an element of $W$. Write $t_h = p_hL_h$.
We can thus apply Corollary \ref{cormain} to see that
\begin{align*}
 t_h(P_a) & = t_h( (1-a)A) - t_h(1) + t_h(aZ) \\
& =\frac{p_h}{L_hAL_h^{-1}(p_h)} L_h(1-a) L_hAL_h^{-1}  -1 + 
\frac{p_h}{L_hZL_h^{-1}(a)} L_hZL_h^{-1}  \nonumber \\
&= \frac{1-c/b}{1-c} C -1 + \frac{c/b-c}{1-c} BC \in \mathcal{I}. \nonumber
\end{align*}
Moreover we find that
\begin{equation*}
 t_h(t_{ab}(t_h(P_a))) = (1-azb/c) A - 1 + \frac{azb/c(1-c/b)}{1-c} AC \in \mathcal{I}
\end{equation*}
A direct calculation now gives that
\[
 P_c = \frac{c (1-c)(c-abz)}{a} P_a - \frac{c^2(1-a)(1-c) }{a} t_h(t_{ab}(t_h(P_a)))
+ az (c-b) (c-1) t_{ab}(t_h(P_a)) \in \mathcal{I}.
\]
Note that we took an $\mathcal{R}$-linear combination of $P_a$, $t_h(t_{ab}(t_h(P_a)))$ and $t_{ab}(t_h(P_a))$ in which
we eleminated the terms containing $A$ and $AC$.

To find $Q_a$ we first calculate
\[
t_h(t_{ab}(t_h(t_{ab}(t_h(P_a)))))= \frac{1-c/a}{1-z} A^{-1} Z - 1 + \frac{c/a(1-abz/c)}{1-z} Z \in \mathcal{I}.
\]
Thus we find
\[
 Q_a = q(a-c) A^{-1}P_a - \frac{a^2}{1-z} t_h(t_{ab}(t_h(t_{ab}(t_h(P_a))))),
\]
(where we use $A^{-1}P_a = (1-a/q)-A^{-1} + a/q A^{-1}Z$).
This immediately also gives $Q_b=t_{ab}(Q_a)\in \mathcal{I}$.

Finally we find that
\[
 R_z = (z-q) Z^{-1} t_h(t_{ab}(t_h(t_{ab}(t_h(P_a))))) - \frac{1}{a} Q_a \in \mathcal{I}. \qedhere
\]
\end{proof}

Now we will show in a few steps that the set of elements in $\mathcal{I}$ from Proposition \ref{prop1}
generate $\mathcal{I}$.

Recall that $\mathcal{D}^*=\mathcal{R}^* \ltimes N$.
\begin{lemma}\label{prop2}
The ideal $\mathcal{I}$ does not contain any units, i.e.\ $\mathcal{I} \cap \mathcal{D}^*=\emptyset$.
Moreover the ideal $\mathcal{I}$ does not contain elements of the form $D_1+D_2$ with $D_1,D_2\in \mathcal{D}^*$ and
$D_1+D_2\neq 0$.
\end{lemma}
\begin{proof}
To show $\mathcal{I}$ does not contain any units, we observe that if $D\in \mathcal{I} \cap \mathcal{D}^*$, then so
is $D^{-1} D=1$. However ${}_2\phi_1$ is not identically zero, thus $1\not \in \mathcal{I}$.

Now suppose there exists some element $D_1 + D_2 \in \mathcal{I}$, with $D_1,D_2\in \mathcal{D}^*$, and $D_1+D_2\neq 0$. 
Left multiplying by $D_2^{-1}$ we find that $D_2^{-1}D_1 +1 \in \mathcal{I}$. Thus $\tilde D:=-D_2^{-1}D_1$ leaves
${}_2\phi_1$ invariant. Recall that $\mathcal{D}^*$ is a subgroup of the group of transformations $T$. 
We can therefore view $\tilde D$ as an element of $T$, and, as it leaves ${}_2\phi_1$ invariant, even 
of $W$. Now we can use the action of $T$ on $\mathcal{D}$ and Corollary \ref{cormain}, to note that 
$\tilde D (R_z) \in \mathcal{I}$. Write $\tilde D=\tilde p \tilde Y$, thus we calculate
\begin{align*}
 \tilde p &\tilde Y (R_z) \\ 
&=  \tilde Y ( az+bz-c-q)  + 
\frac{\tilde p}{\tilde Y Z^{1} \tilde Y^{-1}( \tilde p)} \tilde Y(q-z) Z^{-1}
+ \frac{\tilde p}{\tilde Y Z\tilde Y^{-1}(\tilde p)} \tilde Y (c-abz) Z \\
&= \tilde Y (az+bz-c-q) + \frac{1}{Z^{-1} (r)} \tilde Y (q-z) Z^{-1} + r \tilde Y(c-abz) Z \in \mathcal{I}, 
\end{align*}
where in the last equation we substitute $r=\frac{ \tilde p}{Z \tilde p}$, and we 
use that $\tilde Y, Z$ are elements of the abelian group $N$, so $\tilde Y Z \tilde Y^{-1}=Z$ and 
$\tilde Y Z^{-1} \tilde Y^{-1} = Z^{-1}$.

Now there are two possibilities, either this element is a multiple (over the field $\mathcal{R}$) of
$R_z$, or it is not. If it is, we find that the quotients of the coefficients before $1$, $Z$ and $Z^{-1}$ in this element and
in $R_z$ have to be equal. Thus we get the equations 
\begin{align*}
\frac{\tilde Y(az+bz-c-q)}{az+bz-c-q} =
\frac{ \tilde Y(q-z)}{Z^{-1}(r) (q-z)} =
\frac{r \tilde Y(c-abz)}{c-abz}.
\end{align*}
The equality between the first and last term gives the following expression for $r$
\begin{equation}\label{eq314}
r= \frac{(c-abz) \tilde Y(az+bz-c-q)}{(az+bz-c-q)\tilde Y(c-abz) }.
\end{equation}
The equality between the first and second term gives an expression for $Z^{-1}(r)$ 
\[
Z^{-1}(r) = \frac{(az+bz-c-q) \tilde Y(q-z)}{(q-z) \tilde Y(az+bz-c-q)},
\]
which after letting $Z$ act on both sides of the equation (recall $Z$ and $\tilde Y$ commute) becomes
\begin{equation}\label{eq315}
 r=\frac{(aqz+bqz-c-q)\tilde Y(1-z) }{(1-z) \tilde Y(aqz+bqz-c-q) }.
\end{equation}
Equating the right hand sides of \eqref{eq314} and \eqref{eq315} leads to the expression
\[
 \frac{(c-abz) (1-z) \tilde Y(aqz+bqz-c-q) \tilde Y(az+bz-c-q)}{(az+bz-c-q)(aqz+bqz-c-q)\tilde Y(1-z)\tilde Y(c-abz) } =
1.
\]
As $\tilde Y \in \mathcal{D}^*$ it acts by multiplying the monomials in $a$, $b$, $c$ and $z$ by powers of $q$.
In particular corresponding terms of the numerator and denominator can only differ by powers of $q$. 
For example there exists some $k\in \mathbb{Z}$ such that $q^k (1-z)=\tilde Y(1-z)$. 
As $\tilde Y(1-z) = 1-\tilde Y(z)$,  we have $k=0$ and $Y(z)=z$. Next we find either $\tilde Y(az+bz-c-q)=q^{k_1} (az+bz-c-q)$ and 
$\tilde Y(aqz+bqz-c-q) = q^{k_2}(aqz+bqz-c-q) $ for some integers $k_1$ and $k_2$, or 
$\tilde Y(az+bz-c-q)=q^{k_1} (aqz+bqz-c-q)$ and 
$\tilde Y(aqz+bqz-c-q) = q^{k_2}(az+bz-c-q) $. In both cases we find (looking at the term $q$) that $k_1=k_2=0$, 
and subsequently that $\tilde Y(c)=c$. The second case moreover gives $\tilde Y(a) = aq$ and $\tilde Y(aq)=a$, which 
gives a contradiction, so we can assume the first case holds. This gives $\tilde Y(a)=a$ and $\tilde Y(b)=b$.
We conclude that $\tilde Y=1$, as it leaves $a$, $b$, $c$ and $z$ invariant. However this implies that
$D_2=fD_1$ for some rational function $f$, so $D_1+D_2 \in \mathcal{D}^*$. The first part of the proposition now
shows that $D_1+D_2$ is not an element of $\mathcal{I}$.

Now suppose $\tilde D(R_z)$ is not a multiple of $R_z$. In this case we can find a linear combination 
(over the field $\mathcal{R}$) of $\tilde D(R_z)$ and $R_z$ in which we eliminate the $Z^{-1}$ term. (We would be unable to
cancel the $Z^{-1}$ term if the coefficient before $Z^{-1}$ in either of the equations vanishes identically. However
for $R_z$ we know this is clearly not the case. For $\tilde D(R_z)$ we find that 
$\tilde Y (q-z) /Z^{-1}(r)$ is also non-zero.) In particular we find a non-zero  element $s_1Z + s_2\in \mathcal{I}$ for some
$s_1,s_2\in \mathcal{R}$. Taking a linear combination
of this element and $P_a$ in which we cancel the $Z$-term, we find an element
$t_1A-t_2 \in \mathcal{I}$. (Once again the coefficient in front of $Z$ for $P_a$ is obviously non-zero, and $s_1\neq 0$, as then 
we would have found an element in $\mathcal{I}$ of only one term, which the first part of the proposition shows is impossible). 
Without loss of generality we can assume $t_1$ and $t_2$ are 
polynomials with greatest common divisor 1 (if not multiply (on the left) by the denominators of $t_1$ and $t_2$ and divide by their
common factors). 

Now we need the following lemma:
\begin{lemma}\label{lem1}
 The function $\mfi{q^j}{b}{c}{z}$ for $j\in \mathbb{N}$, $j>0$, is not a rational function 
(i.e.\ element of $\mathbb{C}(b,c,q,z)$). 
\end{lemma}
\begin{proof}
Suppose $\mfi{q^j}{b}{c}{z}$ is a rational function. Then $f(b,q,z) = \mfi{q^j}{b}{q^j}{z}$ is 
also a rational function. Note that $f(b,q,z)={}_1\phi_0(b;-;q,z)=(zb;q)_\infty/(z;q)_\infty$ by the $q$-binomial theorem \cite[(II.3)]{GenR}.
However $(zb;q)_\infty/(z;q)_\infty$ is clearly not a rational function of $b$, $q$ and $z$ (indeed as a function of
$b$ it has infinitely many zeros). This is a contradiction, so $\mfi{q^j}{b}{c}{z}$ can not be a rational function.
\end{proof}
The fact $t_1 A -t_2 \in \mathcal{I}$ implies an equation of the form 
\[
 t_1 \mfi{qa}{b}{c}{z} = t_2 \mfi{a}{b}{c}{z}.
\]
Inserting $a=1$ we find that on the right hand side we obtain a polynomial, while the left hand side is 
a polynomial times a term $\mfi{q}{b}{c}{z}$, which is not a rational function by Lemma \ref{lem1}. Therefore
the equation can only hold if $t_1 |_{a=1} =0$. However then the right hand side has to vanish at $a=1$ as well, and 
as $\mfi{1}{b}{c}{z}=1$, we find $t_2|_{a=1}=0$. This implies that $(a-1) | t_1$ and $(a-1)|t_2$ in contradiction to 
the assumption that $gcd(t_1,t_2)=1$.
\end{proof}

The previous proposition shows that there exist no first order $q$-difference equations satisfied by ${}_2\phi_1$
(i.e.\ elements in $\mathcal{I}$ which are a linear combination of two elements from $\mathcal{D}^*$).
As we have seen in Proposition \ref{prop1} the ${}_2\phi_1$ does satisfy several 
second order $q$-difference equations (elements in $\mathcal{I}$ which are a linear combination of 
three elements from $\mathcal{D}^*$).
The following proposition shows that it even satisfies such an equation for all possible triples of elements from $\mathcal{D}^*$
\begin{prop}\label{prop3}
Let $\mathcal{I}'$ be the left ideal of $\mathcal{D}$ generated by 
$P_a$, $P_b$, $P_c$, $Q_a$, $Q_b$, $Q_z$ and $R_z$.
For any three distinct $X_1, X_2, X_3 \in N$ there exist non-zero polynomials
$p_1$, $p_2$ and $p_3$ such that $p_1X_1+p_2X_2+p_3X_3 \in \mathcal{I}'$. 
The polynomials are unique up to multiplication by elements of $\mathcal{R}$.
\end{prop}
Note that $\mathcal{I}'\subseteq \mathcal{I}$ and indeed Corollary \ref{cor3} shows that $\mathcal{I}'=\mathcal{I}$.
\begin{proof}
First we prove that there exists for each $X\in N$ an element in $\mathcal{I}'$ of the form $p_{x,x} X + p_{z,x} Z +p_{1,x} $, where
$p_{1,x}$, $p_{z,x}$ and $p_{x,x}$ are polynomials, which are not all zero. 
If $X \in N \backslash \{Z,1\}$ this implies that all polynomials
$p_{x,x}$, $p_{z,x}$ and $p_{1,x}$ are non-zero by Proposition \ref{prop2} and the fact that $\mathcal{I}'\subseteq \mathcal{I}$.
We will prove this by induction on the degree of $X$, which is 
defined by $\deg(M_{k_a,k_b,k_c,k_z}) =|k_a|+|k_b|+|k_c|+|k_z|$. 

The statement holds clearly for $X=1$ and $X=Z$, for then we can take $p_{x,1}=1$, $p_{z,1}=0$ and $p_{1,1}=-1$, respectively
$p_{x,z}=1$, $p_{1,z}=0$ and $p_{z,z}=-1$. Moreover the definition of $\mathcal{I}'$ implies that the statement holds for 
all other $X$ of degree 1.

Suppose $X = SY$, for some $S\in \{A^{\pm 1},B^{\pm 1},C^{\pm 1},Z^{\pm 1}\}$ and that
$\deg(X)>\max(1,\deg(Y))$. 
By the induction hypothesis we can assume there exists an element of the form
$E_0:=p_{x,y} Y + p_{z,y} Z + p_{1,y} \in \mathcal{I}'$. 
Moreover from the definition we know that there exists an element
$E_1:=p_{x,s} S + p_{z,s} Z + p_{1,s} \in \mathcal{I}'$ (if $S=Z$ we have $S-Z \in \mathcal{I}'$, in each case we can assume 
$p_{x,s} \neq 0$) 
and that $R_z=p_{x,z^{-1}} Z^{-1} + p_{z,z^{-1}} Z + p_{1,z^{-1}} \in \mathcal{I}'$.
Left multiplying by $S$, respectively $Z$ now shows that
\begin{align*}
 E_2&:=SE_0=S(p_{x,y})  SY + S(p_{z,y}) SZ + S(p_{1,y}) S\in \mathcal{I}',  \\
E_3&:=ZE_1=Z(p_{x,s}) SZ + Z(p_{z,s}) Z^2 + Z(p_{1,s}) Z \in \mathcal{I}', \\ 
E_4&:= ZR_z=Z(p_{x,z^{-1}}) + Z(p_{z,z^{-1}}) Z^2 + Z(p_{1,z^{-1}})Z \in \mathcal{I}'.
\end{align*}
Combining these equations shows
\begin{align*}
& Z(p_{z,z^{-1}})p_{x,s} Z(p_{x,s}) E_2  - Z(p_{z,z^{-1}})p_{x,s} S(p_{z,y}) E_3  
-Z(p_{z,z^{-1}}) S(p_{1,y}) Z(p_{x,s}) E_1\\ & \qquad  + p_{x,s} S(p_{z,y})Z(p_{z,s}) E_4 \\ &
 =
Z(p_{z,z^{-1}}) p_{x,s} S(p_{x,y})Z(p_{x,s}) X \\ & \quad 
+ \left( p_{x,s} S(p_{z,y})Z(p_{z,s}) Z(p_{1,z^{-1}}) -  Z(p_{z,z^{-1}}) p_{x,s} S(p_{z,y}) Z(p_{1,s}) - Z(p_{z,z^{-1}})p_{z,s}S(p_{1,y})Z(p_{x,s}) \right) Z \\ & \quad
+ (p_{x,s} S(p_{z,y})Z(p_{z,s})Z(p_{x,z^{-1}}) - Z(p_{z,z^{-1}}) p_{1,s}S(p_{1,y})Z(p_{x,s})) 
\in \mathcal{I}'
\end{align*}
Note that the coefficient $Z(p_{z,z^{-1}}) p_{x,s} S(p_{x,y})Z(p_{x,s})$ before $X$ is 
non-zero (as product of non-zero polynomials; in the above equations only the polynomials 
$p_{1,y}$ or $p_{1,s}$ could be zero if $Y$ or $S$ equals $Z$). Therefore this gives us an appropriate equation for $X$.

Now let $X_1$, $X_2$ and $X_3\in N$ be arbitrary and distinct. Then we have elements
\begin{align*}
 F_1 &:= p_{x,x_1x_3^{-1}} X_1X_3^{-1} + p_{z,x_1x_3^{-1}} Z + p_{1,x_1x_3^{-1}} \in \mathcal{I}',\\
 F_2 & := p_{x,x_2x_3^{-1}} X_2X_3^{-1} + p_{z,x_2x_3^{-1}} Z + p_{1,x_2x_3^{-1}} \in \mathcal{I}'.
\end{align*}
In particular also 
\begin{align*}
X_3( p_{z,x_2x_3^{-1}}F_1 - p_{z,x_1x_3^{-1}} F_2) & =
X_3( 
p_{z,x_2x_3^{-1}} p_{x,x_1x_3^{-1}} X_1X_3^{-1}  - p_{z,x_1x_3^{-1}} p_{x,x_2x_3^{-1}} X_2X_3^{-1} 
\\ & \qquad + p_{z,x_2x_3^{-1}}p_{1,x_1x_3^{-1}} -p_{z,x_1x_3^{-1}}p_{1,x_2x_3^{-1}}) \\&  =
X_3( 
p_{z,x_2x_3^{-1}} p_{x,x_1x_3^{-1}}) X_1  - X_3(p_{z,x_1x_3^{-1}} p_{x,x_2x_3^{-1}}) X_2
\\ & \qquad + X_3(p_{z,x_2x_3^{-1}}p_{1,x_1x_3^{-1}} -p_{z,x_1x_3^{-1}}p_{1,x_2x_3^{-1}}) X_3 \in \mathcal{I}'.
\end{align*}
Note that 
$X_1X_3^{-1} \neq 1$ and $X_2X_3^{-1} \neq 1$, so $p_{x,x_1x_3^{-1}}, p_{z,x_1x_3^{-1}},p_{x,x_2x_3^{-1}} ,p_{z,x_2x_3^{-1}} \neq 0$.
As both $p_{z,x_2x_3^{-1}}$ and $p_{x,x_1x_3^{-1}}$ are non-zero this implies that the coefficient before $X_1$ in 
this equation is non-zero. Thus by Proposition \ref{prop2} all coefficients in this equation 
are non-zero and this is the desired element of $\mathcal{I}'$.

Uniqueness of the coefficients $p_j$ of an element $D_1=p_1X_1+p_2X_2+p_3X_3 \in \mathcal{I}'$ follows from the fact that if there also exists an 
element $D_2=r_1X_1+r_2X_2+r_3X_3 \in \mathcal{I}'$ the element $D_1-\frac{p_1}{r_1}D_2$ is an element of $\mathcal{I}'$, of the form
$s_2X_2+s_3X_3$ (with $s_j = p_j - p_1r_j/r_1$), and thus vanishes by Proposition \ref{prop2}.
\end{proof}
\begin{cor}\label{cor3}
The ideal $\mathcal{I}$ is generated by $P_a$, $P_b$, $P_c$, $Q_a$, $Q_b$, $Q_z$ and $R_z$ (i.e. $\mathcal{I}=\mathcal{I}'$).
\end{cor}
\begin{proof}
We observe that $\mathcal{I}'\subseteq \mathcal{I}$ so all we have to show is that $\mathcal{I}$ does not contain any more elements.
Define the lenght of an element $\sum_{j} r_j P_j \in \mathcal{D}$, where $r_j \in \mathcal{R}$ and with distinct $P_j \in N$
to be 
the number of different terms in the sum. We will prove by induction on the length of the elements in $\mathcal{D}$ that
there are no elements of length $n$ in $\mathcal{I}$ which are not also in $\mathcal{I}'$. 
Indeed Proposition \ref{prop2} shows that $\mathcal{I}$ does not contain elements of length smaller or equal than 2, so 
there are definitely no elements of length smaller than or equal to 2 which are not also part of $\mathcal{I}'$. 

Suppose all elements of length smaller than $n$ in $\mathcal{I}$ are also elements of $\mathcal{I}'$. 
Moreover suppose $D=\sum_{j=1}^n r_j P_j \in \mathcal{I}$ is an element of length $n$ and $D\not \in \mathcal{I}'$. We know
there exist non-zero polynomials $p_1$, $p_2$ and $p_3$ such that $D'=p_1P_1+p_2P_2+p_3P_3\in \mathcal{I}'$.
Consider the linear combination $D''=p_1D-r_1 D'$. We see that $D''\in \mathcal{I}$ and $D''\not \in \mathcal{I}'$, while
the length of $D''$ is less than $n$ (as it is a linear combination of $P_2, P_3, \ldots,P_n$). This is a contradiction, so 
there exist no elements of length $n$ in $\mathcal{I}$ which are not also an element of $\mathcal{I}'$.

Thus we conclude that the statement of the corollary holds.
\end{proof}
Note that the proof of Proposition \ref{prop3} 
also gives an explicit algorithm to calculate all three term $q$-difference equations
satisfied by ${}_2\phi_1$. However, the resulting calculations become tedious even for small degrees of $X_1,X_2,X_3 \in N$.

Now we determine some properties of the explicit coefficients $p_j$ occuring in these $q$-difference equations.
\begin{prop}\label{propwhatterms}
Suppose we have an element $p_1X_1+p_2X_2+p_3X_3 \in \mathcal{I}$, where $p_1,p_2,p_3$ are polynomials 
with no common divisors and $X_1,X_2,X_3 \in N$. Let $X_j=M_{k_{a,j},k_{b,j},k_{c,j},k_{z,j}}$ for some
$k_{a,j}, k_{b,j}, k_{c,j}, k_{z,j} \in \mathbb{Z}$ and $j=1,2,3$.
Moreover assume that $k_{a,1}> k_{a,2} > k_{a,3}$.
Then $a-q^{-j} ~|~ p_1$ for $k_{a,2}\leq j < k_{a,1}$, and $a-q^{-j} ~ \not | ~ p_1 $ for $k_{a,3} \leq j <k_{a,2}$.
Moreover $a-q^{-j} ~\not |~ p_2$ for $k_{a,3} \leq j < k_{a,1}$ and
$a-q^{-j} ~\not |~ p_3$ for $k_{a,2} \leq j< k_{a,1}$. 

A similar result holds for $b$ and $a$ interchanged.
\end{prop}
\begin{proof}
Consider the equation $(p_1X_1+p_2X_2+p_3X_3)  {}_2\phi_1(a,b;c;q,z) =0$. 
If we insert $a=q^{-j}$ for some $k_{a,2} \leq j<k_{a,1}$ then the series
$X_2\ {}_2\phi_1(a,b;c;q,z) (={}_2\phi_1(aq^{k_{a,2}},bq^{k_{b,2}};cq^{k_{c,2}};q,zq^{k_{z,2}}) )$ and $X_3\ {}_2\phi_1(a,b;c;q,z)$ terminate (i.e. they are polynomial), 
while $X_1\ {}_2\phi_1(a,b;c;q,z)$ does not.
Indeed by Lemma \ref{lem1} $X_1\ {}_2\phi_1(a,b;c;q,z)|_{a=q^{-j}}$ is not even a rational function of 
$b,c,q$ and $z$. However $p_1 X_1\ {}_2\phi_1(a,b;c;q,z)|_{a=q^{-j}}$ clearly is a rational function. This can 
only hold if $p_1 |_{a=q^{-j}} =0$, so if $a-q^{-j} ~|~ p_1$.

Now suppose $(a-q^{-j})$ is also a factor of $p_2$ or $p_3$ 
(it cannot be both as $p_1$, $p_2$ and $p_3$ have no common divisors). In this case we 
find that one non-zero polynomial equals zero, which is impossible. Thus $(a-q^{-j})$ does not divide
$p_2$ or $p_3$.

If on the other hand we insert $a=q^{-j}$ for some $k_{a,3} \leq j < k_{a,2}$ then $X_3\ {}_2\phi_1(a,b;c;q,z)$ terminates, while the
other two series $X_1\ {}_2\phi_1(a,b;c;q,z)$ and $X_2\ {}_2\phi_1(a,b;c;q,z)$ don't. If the polynomials $p_1$ or 
$p_2$ now vanish for this value of $a$, we again find that the other series should either be a rational function, which 
it is not, or the polynomial in front of that term (respectively $p_2$ or $p_1$) should also vanish for this value of $a$. 
However if both $p_1$ and $p_2$ vanish for this value of $a$, then so should $p_3$, which is impossible as they have no
common divisors. Thus neither $p_1$ nor $p_2$ can be divisible by $a-q^{-j}$ for these values of $j$.
\end{proof}

\section{Proof of the main theorem}\label{secproofmain}
This section contains the proof of Theorem \ref{thmain}. The main ingredients of the proof 
are Corollary \ref{cormain} and Proposition \ref{propwhatterms}.

\begin{proof}[Proof of Theorem \ref{thmain}]
We begin by proving the following important lemma
\begin{lemma}
Suppose $pL \in W$ (recall this implies $p\in H$ and $L\in G$). 
Call $Y=L^{-1}ZL$ (so $Y\in N$) and let $f Y + g  + h Y^{-1} \in \mathcal{I}$ 
(note there exists such an element due to Proposition \ref{prop3}).
Then the following equality between rational functions
\begin{equation}\label{eqjan3}
\frac{(az+bz-c-q)(aqz+bqz-c-q)}{(c-abz)q(1-z)} = 
L \left( \frac{g Y(g)}{f Y(h)} \right)
\end{equation}
holds.
\end{lemma}
\begin{proof}
By Corollary \ref{cormain} we find that $pL(f Y + g  + h Y^{-1}) \in \mathcal{I}$. Explicitly we have
\begin{align*}
 pL(f Y + g  + h Y^{-1}) & = 
\frac{p}{LYL^{-1}(p)} L(f)  LYL^{-1} + L(g) + \frac{p}{LY^{-1}L^{-1}(p)} L(h) LY^{-1}L^{-1} \\
&= \frac{p}{Z(p)} L(f)  Z + L(g) + \frac{p}{Z^{-1}(p)} L(h) Z^{-1} 
\end{align*}
as $LYL^{-1}=L(L^{-1}ZL) L^{-1}=Z$. 
By Proposition \ref{prop3} and Corollary \ref{cor3} there exists, up to left-multiplication by rational functions, 
just one element in $\mathcal{I}$ of the form $\cdot Z + \cdot + \cdot Z^{-1}$, in particular 
the element $R_z$ from Proposition \ref{prop1}. 
Thus $pL(f Y + g  + h Y^{-1}) $ must be a multiple of $R_z$.
We obtain the equations
\[
\frac{pL(f)}{Z(p) (c-abz)} = \frac{L(g)}{(az+bz-c-q)} = \frac{pL(h)}{Z^{-1}(p) (q-z)}.
\]
The first equation implies
\begin{equation}\label{eqjan1}
 \frac{Z(p)}{p} = \frac{(az+bz-c-q)L(f)}{(c-abz)L(g)},
\end{equation}
while the second equation gives
\[
 \frac{p}{Z^{-1}(p)} = \frac{(q-z)L(g)}{(az+bz-c-q) L(h)}.
\]
After applying $Z$ on both sides of the latter equation this becomes
\begin{equation}\label{eqjan2}
 \frac{Z(p)}{p} = \frac{q(1-z)ZL(g)}{(aqz+bqz-c-q) ZL(h)}.
\end{equation}

Combining \eqref{eqjan1} and \eqref{eqjan2} gives
\[
\frac{(az+bz-c-q)L(f)}{(c-abz)L(g)} =
\frac{q(1-z)ZL(g)}{(aqz+bqz-c-q) ZL(h)},
\]
which can be reduced to \eqref{eqjan3} by using $ZL = LY$.
\end{proof}
The previous lemma shows that the coefficients $f$, $g$ and $h$ in the difference equation $fY+g+hY^{-1}$ are 
not very difficult. Indeed, applying a operator such as $L$ does not increase the number of terms of some polynomial, and
retains any factorization. We can use this to obtain restrictions on the possible $pL \in W$.
In particular we reduce the number of possible elements $L^{-1}ZL$.
\begin{lemma}\label{propfirst}
 Suppose $pL\in W$, and $L^{-1}ZL:=Y=M_{k_a,k_b,k_c,k_z}$, then $|k_a|,|k_b| \leq 1$.
\end{lemma}
\begin{proof}
Suppose $k_a >1$. Let $fY+g+hY^{-1} \in \mathcal{I}$ and assume $f$, $g$ and $h$ are non-zero polynomials with $gcd(f,g,h)=1$.

Using Proposition \ref{propwhatterms} we find that $a-q^{-j}~|~f$ for $0\leq j<k_a$, and $a-q^{-j} \not | f$ for $-k_a\leq j<0$. 
Moreover $a-q^{-j}\not | g$ for $-k_a\leq j<k_a$. 
This implies in particular that $a-1$ and  $a-q^{-1}$ both divide $f$, while
they do not divide $g$. Moreover note that $Y(a-q^{-j})=aq^{k_a} -q^{-j} = q^{k_a}(a-q^{-j-k_a})$. 
In particular we also find that $a-1$ and $a-q^{-1}$ also do not divide $Y(g)$, as 
$a-q^{k_a}$ and $a-q^{k_a-1}$  do not divide $g$. This implies that the 
expression $gY(g)/fY(h)$ contains the factors $(a-1)(a-q^{-1})$ in the denominator when expressed as $x/y$ with $gcd(x,y)=1$. 
After applying $L$ to it, we find that the denominator of $L(gY(g)/fY(h))$ should still
contain a factor of the form $(x-1)(x-q^{-1})$, for some monomial $x \in \mathbb{C}[a^{\pm 1},b^{\pm 1},c^{\pm 1},q^{\pm 1},z^{\pm 1}]^*$.
However the left hand side of \eqref{eqjan3} clearly does not contain such a factor in the denominator, which gives a contradiction.

The proof for $|k_b|\leq 1$ is completely similar.
\end{proof}

We can now use the known symmetries to further decrease the possible values of $L^{-1} ZL$. Indeed we have
\begin{lemma}
 Suppose $W=pL\in W$, and $L^{-1}ZL:=Y=M_{k_a,k_b,k_c,k_z}$, then $k_a$, $k_b$, $k_c$ and 
$k_z$ are one of the following combinations, or one of them with $k_a \to -k_a$ or with $k_a\leftrightarrow k_b$:
\begin{center}
\begin{tabular}{cccc|cccc|cccc|cccc|cccc}
 $k_a$ & $k_b$ & $k_c$ & $k_z$ & $k_a$ & $k_b$ & $k_c$ & $k_z$& $k_a$ & $k_b$ & $k_c$ & $k_z$& $k_a$ & $k_b$ & $k_c$ & $k_z$& $k_a$ & $k_b$ & $k_c$ & $k_z$\\
\hline
1 & 1 & 2 & 1 &         1 & 1 & 1 & -1 &      1 & 0 & 1 & -1 &      1 & -1 & 0 & 0 &     0 & 0 & 1 & 1 \\
1 & 1 & 2 & 0 &         1 & 1 & 0 & -1 &      1 & 0 & 0 & 0 &       1 & -1 & 0 & -1 &    0 & 0 & 0 & 1 \\
1 & 1 & 2 & -1 &        1 & 0 & 1 & 1 &       1 & 0 & 0 & -1 &      1 & -1 & 0 & -1 &    &&&\\
1 & 1 & 1 & 0  &        1 & 0 & 1 & 0 &       1 & -1 & 0 & 1 &      1 & -1 & -1 & 0 &    &&&
\end{tabular}
\end{center}

\end{lemma}
\begin{proof}
We know that Heine's transformation $t_h\in W$ and the 
$a\leftrightarrow b$ interchanging symmetry $t_{ab}\in W$.
Hence we find that $t_hW, t_{ab}t_h W, \ldots, T_{h}t_{ab}t_ht_{ab}t_h W \in W$.
For all of these operators, we can now apply Proposition \ref{propfirst}, to obtain 
inequalities for $k_a$, $k_b$, $k_c$ and $k_z$ given by 
\begin{multline*}
 |k_a|\leq 1, \quad |k_b|\leq 1, \quad |k_z|\leq 1, \quad |k_b-k_c|\leq 1, \quad  |k_a-k_c|\leq 1, \quad |k_a+k_b-k_c+k_z|\leq 1. 
\end{multline*}
The list now gives all solutions to these inequalities, where we ommitted the cases $k_a<0$ and $k_a<k_b$ for symmetry 
reasons. Note that we omitted $(0,0,0,0)$ as $L^{-1}ZL \neq 1$.
\end{proof}

Since this is a finite set we can exactly calculate the expressions $gY(g)/fY(h)$ for all corresponding $Y$, and see
which are such that the relation \eqref{eqjan3} can hold (using again that the factorization of such terms is
not effected by applying the operator $L$). This leads to the following lemma.
\begin{lemma}
 Suppose $pL\in W$ and let $Y=LZL^{-1}$, then $Y=Z$ or $Y=AC$ or $Y=BC$.
\end{lemma}
\begin{proof}
 We calculate $gY(g)/fY(h)$ for $Y=A^{k_a}B^{k_b}C^{k_c}Z^{k_z}$ for all possibilities of the last
proposition. Then we throw out all cases in which the denominator (after writing it as $x/y$ for some
polynomials $x$ and $y$ with $gcd(x,y)=1$) can not be written as the product of two terms which are both the sum of 
two monomials. The calculations are extremely tedious and were therefore performed by computer.
\end{proof}

We need to further disallow a few possible transformations. Indeed so far all the conditions we set were 
derived from the fact that if $t\in W$, then $t({}_2\phi_1)$ satisfies the same $q$-difference equations
as ${}_2\phi_1$. However as these are second order $q$-difference equations, there exists a second
independent solution. This solution can also be expressed in terms of ${}_2\phi_1$, so we must give another
argument why it is independent of ${}_2\phi_1$.
\begin{lemma}\label{lemimp}
There exist no element $pL\in W$, where 
\[
 L=\left( \begin{array}{ccccc}
           1 & 0 & -1 & 0 & 1 \\
           0 & 1 & -1 & 0 & 1 \\
           0 & 0 & -1 & 0 & 2 \\
           0 & 0 &  0 & 1 & 0 \\ 
           0 & 0 & 0 &  0 & 1
          \end{array} \right)
\]
and $p\in H$ is arbitrary.
\end{lemma}
\begin{proof}
Suppose such an equation does exist.
As $P_a\in \mathcal{I}$ we find by Corollary \ref{cormain} that 
$pL(P_a) \in \mathcal{I}$. An explicit calculation now gives (observe that $L^2=1$, so $L=L^{-1}$)
\[
pL(P_a)  =
pL( (1-a) A)L^{-1}p^{-1} - 1 + pL(aZ)L^{-1} p^{-1}
= \frac{p}{A(p)} (1-aq/c)  A -1 + \frac{p}{ Z(p)} \frac{aq}{c} Z
\]
This element should be a multiple of $P_a$ (due to Proposition \ref{prop3}), and as the constant term is identical we find
\[
\frac{p}{A(p)} = \frac{1-a}{1-aq/c}, \qquad \frac{p}{Z(p)}= \frac{c}{q}.
\]
Similarly using $P_b$ instead of $P_a$ we find $p/B(p)=(1-b)/(1-bq/c)$, and using $Q_c$ we find
\begin{equation}\label{eqmult1}
 pL(Q_c) = 
-q+pL(cZ)L^{-1}p^{-1} + pL(q-c)C^{-1}L^{-1}p^{-1} 
 = -q + \frac{p}{Z(p)} \frac{q^2}{c}Z + \frac{p}{ABC(p)} (q-\frac{q^2}{c}) ABC
\end{equation}
Now from 
\begin{align*}
 {}_2\phi_1(aq,bq;cq;q,z) & = 
\sum_{k\geq 0} \frac{(aq,bq;q)_k}{(q,cq;q)_k} z^k \\ &= 
\frac{(1-c)}{z(1-a)(1-b)} \sum_{k\geq 0} \frac{(a,b;q)_{k+1}}{(q,c;q)_{k+1}} (1-q^{k+1}) z^{k+1} \\ & = 
\frac{(1-c)}{z(1-a)(1-b)} \left( {}_2\phi_1(a,b;c;q,z) - {}_2\phi_1(a,b;c;q,qz) \right),
\end{align*}
where the last equality holds as the term within the sum vanishes for $k=-1$, we find that
\[
 ABC + \frac{1-c}{z(1-a)(1-b)} (Z-1) \in \mathcal{I}.
\]
Comparing this with \eqref{eqmult1} we find the equality
\[
\frac{p}{ABC(p)}= \frac{qz(1-b)(1-a}{(q-q^2/c)(1-c)}
\]
Now we can find 
\begin{multline*}
 \frac{p}{C(p)}= \frac{p}{ABC(p)} BC(\frac{A(p)}{p}) C(\frac{B(p)}{p}) 
= \frac{z(1-b)(1-a}{(1-q/c)(1-c)} \frac{1-a/c}{1-a} \frac{1-b/c}{1-b} \\ =
\frac{z(1-a/c)(1-b/c)}{(1-q/c)(1-c)} = \frac{z(c-a)(c-b)}{c(c-q)(1-c)}.
\end{multline*}
As the function
\[
 g(a,b;c;q,z) = \frac{(c/a,c/b,q^2/c;q)_\infty}{(c,q/a,q/b;q)_\infty} \frac{\theta(ab,z;q)}{\theta(c/ab,c/z;q)},
\]
where $\theta(x;q)=(x,q/x;q)_\infty$, 
satisfies the equations $Ag/g=\frac{(c-aq)}{c(1-a)}$, $Bg/g=\frac{(c-bq)}{c(1-b)}$, $Cg/g=\frac{c(c-q)(1-c)}{z(c-a)(c-b)}$
and $Zg/g=q/c$ we find that $p=g h$ for some elliptic function $h$ (i.e. $Ah=h$, $Bh=h$, $Ch=h$, and $Zh=h$).
Thus a desired equation is of the form
\begin{equation}\label{eqfeb1}
 {}_2\phi_1(a,b;c;q,z) = \frac{(c/a,c/b,q^2/c;q)_\infty}{(c,q/a,q/b;q)_\infty} \frac{\theta(ab,z;q)}{\theta(c/ab,c/z;q)}
h(a,b,c,q,z) {}_2\phi_1(aq/c,bq/c;q^2/c;q,z).
\end{equation}
Now multiplying this equation by $(1-q/a)$ and 
subsequently inserting $a=q$ we find
\begin{multline*}
 0=(1-q/q) {}_2\phi_1(q,b;c;q,z) \\ = \frac{(c/q,c/b,q^2/c;q)_\infty}{(c,q,q/b;q)_\infty} \frac{\theta(bq,z;q)}{\theta(c/bq,c/z;q)}
h(q,b,c,q,z) {}_2\phi_1(q^2/c,bq/c;q^2/c;q,z)
\end{multline*}
therefore $h(q,b,c,q,z)=0$ (the other terms on the right hand side do not vanish). 
On the other hand inserting $a=1$ in \eqref{eqfeb1} we find
\[
1= {}_2\phi_1(1,b;c;q,z) = \frac{(c,c/b,q^2/c;q)_\infty}{(c,q,q/b;q)_\infty} \frac{\theta(b,z;q)}{\theta(c/b,c/z;q)}
h(1,b,c,q,z) {}_2\phi_1(q/c,bq/c;q^2/c;q,z),
\]
thus $h(1,b,c,q,z)\neq 0$. However since $h$ is elliptic in $a$ we find $0=h(q,b,c,q,z)=h(1,b,c,q,z)\neq 0$, which is a contradiction.
Therefore no equation of the desired form exists.
\end{proof}

Note that the right hand side of \eqref{eqfeb1} indeed satisfies the same $q$-difference equations as
${}_2\phi_1$ itself. (This can be shown by proving it is in the kernel of $P_a$, $P_b$, etc.)

Now we only have to check the possible $L\in G$ which can occur for $Y=Z$, $Y=AC$ and $Y=BC$.
If $Y=L^{-1}ZL=BC$, we can consider $L'=LL_h$ 
(where $L_h\in G$ occurs in Heine's transformation, from \eqref{eqelw}) and find that $Y'=L'^{-1} ZL'=L_h^{-1} Y L_h = 
L_h^{-1} BC L_h = Z$. So without loss of generality we may assume $Y\neq BC$, and similarly $Y\neq AC$.
So let us assume $Y=Z$.

Note that for the case $Y=Z$ we have 
$gY(g)/fY(h)$ equal to the left hand side of \eqref{eqjan3}, so we must find $L\in G$ such that
\begin{equation}\label{eqnow}
\frac{(az+bz-c-q)(aqz+bqz-c-q)}{(c-abz)q(1-z)} = 
L \left(\frac{(az+bz-c-q)(aqz+bqz-c-q)}{(c-abz)q(1-z)} \right).
\end{equation}
This implies that we have to look for operators $L$ which preserve the left hand side of the above equation.
Moreover we know that $LZL^{-1}=Z$, so we have $ZL(z)= LZ(z)=L(qz)=qL(z)$. Hence we find
$L(z)=Kz$ for some monomial $K \in \mathbb{Z}[a^{\pm 1},b^{\pm 1},c^{\pm 1},q^{\pm 1}]^*$. Similarly we can prove that 
$L(a), L(b), L(c)$ are all monomials in $\mathbb{Z}[a^{\pm 1},b^{\pm 1},c^{\pm 1},q^{\pm 1}]^*$. 

Now rewrite the left hand side of \eqref{eqnow} as
\[
 \frac{ ( \frac{a+b}{c+q}z-1)(\frac{a+b}{c+q}qz-1)}{(\frac{ab}{c}z-1)(z-1)} \cdot \frac{(q+c)^2}{cq},
\]
i.e. as the unique expression of the form $K (A_1z-1)(A_2z-1)/(B_1z-1)(B_2z-1)$, with coefficients $K,A_1,A_2,B_1,B_2 \in \mathbb{C}(a,b,c,q)$.
As $L$ preserves the constant 1, we find that $L$ has to either interchange the two 
factors in the numerator and denominator, or
leave them invariant. For the numerator we have only one choice as the term which is $q$ times the other term, remains
such after application of $L$. So we have $L((a+b)z/(c+q))=(a+b)z/(c+q)$. Moreover the remaining constant 
has to be preserved, so $L( (q+c)^2/qc) = (q+c)^2/qc$. In the denominator we find that either $L(z)=z$ 
(and thus $L(abz/c)=abz/c$) or $L(z) = abz/c$ (and $L(abz/c)=z$). If $L(z)=abz/c$ we can consider 
$L'=(L_{ab}L_h)^3 L$ instead of $L$, which then satisfies $L'(z)=z$ (and $L'(abz/c)=abz/c$). Thus without loss
of generality we assume $L(z)=z$.
From the equation $L( (q+c)^2/qc) = (q+c)^2/qc$, we find $L(c)/q+q/L(c) = c/q+q/c$, which means that either
$L(c) = c$ or $L(c)=q^2/c$. 

In the case $L(c)=q^2/c$ we find
\[L(a+b) =L( \frac{(a+b)z}{(c+q)}) L( \frac{c+q}{z}) = \frac{(a+b)(q^2/c+q)}{c+q} =\frac{(a+b)q}{c}\] and
$L(ab) = L(abz/c) L(c/z) = abq^2/c^2$. Thus we find the polynomial identity $(x-L(a))(x-L(b))=x^2-L(a+b)x+L(ab)=(x-aq/c)(x-bq/c)$.
From this we conclude that either $L(a)=aq/c$ and $L(b)=bq/c$ or vice versa. 
However Lemma \ref{lemimp} shows that there exist no elements in $W$ with this $L$. Thus we can assume from now on that $L(c)=c$. 

Now as $L(z)=z$ we find that $L(a+b) = L( (a+b)z/(c+q)) L( (c+q)/z) = (a+b)$, and 
$L(ab) = L(abz/c) L(c/z) = ab$. Thus we find that $(x-L(a))(x-L(b))=(x-a)(x-b)$. Therefore either
$L(a)=a$ and $L(b)=b$ or $L(a)=b$ and $L(b)=a$. In the first case we end up with the identity, and in the second case with the 
$a\leftrightarrow b$ shifting symmetry $t_{ab}$.
Note that each $L$ can lead to at most one transformation, as two different transformations corresponding to the 
same $L$ would only differ in a $q$-hypergeometric multiplicative term, which then clearly has to equal 1.

Thus there are no other elements in $W$ than those in the group generated by $t_h$ and $t_{ab}$.
\end{proof}

\end{document}